\documentclass[11pt, reqno]{amsart}
\usepackage[all]{xy}

\usepackage{amssymb}
\usepackage{amsthm}
\usepackage{amsmath,cancel}
\usepackage{amscd,enumitem}
\usepackage{verbatim}
\usepackage{float}
\usepackage{geometry}
\usepackage{color,mdframed}
\usepackage{bbm}
\usepackage[mathscr]{eucal}
\usepackage[all]{xy}
\usepackage{bm}

\usepackage{a4wide,fullpage}
\usepackage[draft]{hyperref}
\usepackage{mathtools}
\usepackage{calligra}
\usepackage{stmaryrd}
\DeclareMathAlphabet{\mathcalligra}{T1}{calligra}{m}{n}

\newtheorem{theorem}{Theorem}
\newtheorem{corollary}[theorem]{Corollary}
\newtheorem{lemma}[theorem]{Lemma}

\newtheorem*{proposition*}{Proposition}
\newtheorem*{theorem*}{Theorem}

\theoremstyle{definition}
\newtheorem*{remark}{Remark}

\theoremstyle{remark}

\newcommand{\R}{\mathbb{R}}
\newcommand{\Q}{\mathbb{Q}}
\newcommand{\Z}{\mathbb{Z}}
\newcommand{\N}{\mathbb{N}}
\newcommand{\C}{\mathbb{C}}

\renewcommand{\H}{\mathbb{H}}

\newcommand{\ord}{{\text {\rm ord}}}

\newcommand{\SL}{{\text {\rm SL}}}

\newcommand{\pihol}{\pi_{\operatorname{hol}}}

\renewcommand{\pmod}[1]{\  \,  \left(  \operatorname{mod} \,  #1 \right)}

\DeclareMathAlphabet{\mathpzc}{OT1}{pzc}{m}{it}

\numberwithin{equation}{section}
\numberwithin{theorem}{section}
\allowdisplaybreaks

\begin{document}
	
	\vspace*{-1cm}
	
\title{Formulas for moments of class numbers in arithmetic progressions}
\author{Kathrin Bringmann}
\address{University of Cologne, Department of Mathematics and Computer Science, Weyertal 86-90, 50931 Cologne, Germany}
\email{kbringma@math.uni-koeln.de}
\author{Ben Kane}
\address{Department of Mathematics, University of Hong Kong, Pokfulam, Hong Kong}
\email{bkane@hku.hk}
\author{Sudhir Pujahari}
\address{Sudhir Pujahari, Institute of Mathematics, Warsaw University, Banacha 2, 02-97 Warsaw, Poland}
\email{spujahari@mimuw.edu.pl}
\thanks{The research of the  first author is supported by the Alfried Krupp Prize for Young University Teachers of the Krupp foundation and has received funding from the European Research Council (ERC) under the European Union's Horizon 2020 research and innovation programme (grant agreement No. 101001179). The research of the second author was supported by grants from the Research Grants Council of the Hong Kong SAR, China (project numbers HKU 17301317, and 17303618).}
\date{\today}
\begin{abstract}
In this paper, we obtain explicit formulas for the second moments for Hurwitz class numbers $H(4n-t^2)$ with $t$ running through a fixed congruence class modulo $3$.
\end{abstract}
\maketitle

\section{Introduction and statement of results}

 Let $\mathcal{Q}_D$ denote the set of integral binary quadratic forms of discriminant $D<0$. The \begin{it}$|D|$-th Hurwitz class number\end{it} is defined by
\[
H(|D|):=\sum_{Q\in\mathcal{Q}_{D}/\SL_2(\Z)} \frac{1}{\omega_Q},
\]
where $\omega_Q$ is half the size of the stabilizer group $\Gamma_Q$ of $Q$ in $\SL_2(\Z)$. By convention, we set $H(0):=-\frac{1}{12}$ and $H(r):=0$ for $r\notin\N_0$ or $r\equiv 1,2\pmod{4}$. 

There are several formulas for sums of class numbers. For example, for a prime $p$ we have the well-known identity (see \cite[p. 154]{Eichler})
\begin{equation}\label{eqn:Kronecker}
\sum_{t\in\Z} H\left(4p-t^2\right)=2p.
\end{equation}
This formula has been useful in a number of settings, including applications to divisibility and non-divisibility of class numbers \cite{Hartung,KohnenOno} and Iwasawa invariants \cite{Horie}.
We obtain generalizations of \eqref{eqn:Kronecker} for special cases of sums of the type
\begin{equation}\label{eqn:HmMkdef}
H_{\kappa,m,M}(n):=\sum_{\substack{t\in\Z\\t\equiv m \pmod{M}}}t^{\kappa} H \left(4n-t^2\right).
\end{equation}
If $\kappa=0$, then we omit it in the notation throughout, writing for example $H_{m,M}(n)$ for $H_{0,m,M}(n)$. Such sums were considered in a few different places. For example, in \cite{BrownCalkin} for a prime $p\equiv 3\pmod{5}$ the following identity was proven: 
\[
H_{1,5}(p)=\begin{cases} 
\frac{1}{3}(p+1)&\text{if }p\equiv 1,2\pmod{5},\\
\frac{1}{2}(p-1) &\text{if }p\equiv 3\pmod{5},\\ 
\frac{5}{12}(p+1)&\text{if }p\equiv 4\pmod{5}.  
\end{cases}
\]
The remaining cases were conjectured; this conjecture was later proven in \cite{BKClassNum}. In this paper, we obtain an analogous formula for the second moments $H_{2,m,3}(n)$. To state the result, let $\sigma(n):=\sum_{d\mid n} d$ denote the \textit{sum of divisors function}, $\delta_S := 1$ if a statement $S$ is true and 0 otherwise, and for $\eta(\tau):=q^{\frac{1}{24}}\prod_{n\geq 1} \left(1-q^n\right)$ (with $q:=e^{2\pi i \tau}$) we write 
\[
\eta(3\tau)^8 =: \sum_{n=1}^{\infty} a(n)q^n.
\]

\begin{theorem}\label{prop:H1,m,3}
For $m\in\{0,1,2\}$, the following hold.
\begin{enumerate}[leftmargin=*, label=\rm(\arabic*)]
\item If $n\not\equiv 0\pmod{3}$, then we have 
\begin{multline*}
\hspace{.33cm} H_{2,m,3}(n)=-\delta_{n=\square}\delta_{m\neq 0} \frac{ n^{\frac{3}{2}}}{2}- \sum_{d\mid n, d^2 <n} d^3
\begin{cases}
2\delta_{n\equiv2\pmod 3}  &\text{if }m=0,\\
\delta_{n\equiv 1 \pmod 3} &\text{if }m\neq 0,
\end{cases}
\\
+\begin{cases}\frac{1}{2}n\sigma(n)-\frac{1}{2}a(n)&\text{if }m=0,\ n\equiv 1\pmod{3},\\
n\sigma(n)-2n\sum_{d\mid n, d^2<n} d&\text{if }m=0,\ n\equiv 2\pmod{3},\\
\frac{3}{4}n\sigma(n)-\frac{n}{2}\sum_{d\mid n} \min\left(d,\frac{n}{d}\right)+\frac{1}{4}a(n)&\text{if }m\neq 0,\ n\equiv 1\pmod{3},\\
\frac{1}{2}n\sigma(n)&\text{if }m\neq 0,\ n\equiv 2\pmod{3}.
\end{cases}
\end{multline*}

\item For $n\in\N$ with $3 \mid n$ we have 
\begin{multline*}
\hspace{.33cm} H_{2,m,3}(n)=-\delta_{n=\square}\delta_{m=0} n^{\frac{3}{2}} -
\begin{cases}
54\delta_{9\mid n}{\displaystyle  \sum_{d\mid \frac{n}{9}, d^2<\frac{n}{9}}} d^3 &\text{if }m=0,\\
\displaystyle \sum_{d\mid n, d^2<n} d^3 - \sum_{\substack{d\mid n, d^2<n\\ 3\mid d,\ 3\mid \frac{n}{d}}} d^3&\text{if }m\neq 0,
\end{cases}\\
+ \begin{cases} \displaystyle 2n\sigma\left(\frac{n}{3}\right)-6n\delta_{9\mid n}{\displaystyle \sum_{d\mid \frac{n}{9}, d^2<\frac{n}{9}}} d -n^{\frac{3}{2}}\delta_{n=\square}&\text{if }m=0,\\ \displaystyle
n\left(\sigma(n)-\sigma\left(\frac{n}{3}\right)\right)-\frac{n}{2}{\displaystyle\sum_{d\mid n}}\min\left(d,\frac{n}{d}\right)+3n\delta_{9\mid n}{\displaystyle\sum_{d\mid \frac{n}{9}, d^2<\frac{n}{9}}} d +\frac{n^{\frac{3}{2}}}{2}\delta_{n=\square}&\text{if }m\neq0.
\end{cases}
\end{multline*}
\end{enumerate}
\end{theorem}

\begin{remark}
Specializing Theorem \ref{prop:H1,m,3} to the case where $n$ is a prime power yields formulas similar to \eqref{eqn:Kronecker} for second moments (see Corollary \ref{cor:H1,m,3primepower}). For example, if $p\equiv 2\pmod{3}$, then 
\[
H_{2,1,3}(p)=\frac{p(p+1)}{2}.
\]
\end{remark}

The paper is organized as follows. In Section \ref{sec:prelim}, we recall some preliminaries about non-holomorphic modular forms and their growth towards cusps and evaluations of generalized quadratic Gauss sums that naturally occur when investigating this growth. In Section \ref{sec:setup}, we then determine the growth towards cusps for specific forms that are related to moments of class numbers, giving the projection of these forms into the space of Eisenstein series in Corollary \ref{cor:Mertensholproj}. We complete the paper by applying Corollary \ref{cor:Mertensholproj} in the special case $M=3$ to prove Theorem \ref{prop:H1,m,3} in Section \ref{sec:explicitmoments}.

\section{Preliminaries}\label{sec:prelim}

\subsection{Non-holomorphic modular forms}

 Let $\mathfrak S$ denote the set of pairs $(\gamma,\varepsilon)$ where $\gamma= \left(\begin{smallmatrix}a&b\\c&d\end{smallmatrix}\right) \in \SL_2(\R)$ and $\varepsilon\colon\H\to \C$ is holomorphic and 
satisfies $|\varepsilon(\tau)|=\sqrt{|c\tau+d|}$. For $\kappa\in\frac{1}{2}\Z$ and a function $F:\H\to\C$, we define the \begin{it}weight $\kappa$ slash operator\end{it} of $(\gamma,\varepsilon)\in \mathfrak{S}$ by  
\[
F|_{\kappa}(\gamma,\varepsilon)(\tau):=\varepsilon(\tau)^{-2\kappa}F(\gamma \tau).
\]
We set
\[
\varepsilon_{\kappa,\gamma}(\tau):=\begin{cases} \sqrt{c\tau+d}&\text{if }\kappa\in\Z,\\ \left(\frac{c}{d}\right) \varepsilon_d^{-1} \sqrt{c\tau+d}&\text{if }\kappa\in \Z +\frac{1}{2}\text{ and }\gamma\in\Gamma_0(4),\end{cases}
\]
where $\varepsilon_d:=1$ for $d\equiv 1\pmod{4}$ and $\varepsilon_d:=i$ for $d\equiv 3\pmod{4}$. 
 We say that $F$ satisfies \textit{weight $\kappa$ modularity on}
 $\Gamma\subseteq\SL_2(\Z)$ if for all $\gamma=\left(\begin{smallmatrix}a&b\\ c&d\end{smallmatrix}\right)\in\Gamma$ 
(if $\kappa\notin\Z$, then we assume that  $\Gamma\subseteq\Gamma_0(4)$) 
\[
F\big|_{\kappa}(\gamma,\varepsilon_{\kappa,\gamma})= F. 
\]
The space of all real-analytic functions satisfying weight $\kappa$ modularity on $\Gamma$ we denote by $\mathcal{M}_{\kappa}(\Gamma)$. For $\Gamma\subseteq\SL_2(\Z)$, we call the elements of $\Gamma\backslash (\Q\cup \{i\infty\})$ the  \begin{it}cusps of $\Gamma$\end{it}. For each cusp (representative) $\varrho=\frac{a}{c}\in\Q\cup\{i\infty\}$, we choose $M_{\varrho}=\left(\begin{smallmatrix} a&b\\ c&d\end{smallmatrix}\right)\in\SL_2(\Z)$. 
For $F\in\mathcal{M}_{\kappa}(\Gamma)$ we call, writing $\tau=u+iv$ throughout, 
\[
F_{\frac ac}(\tau):=(c\tau+d)^{-\kappa} F\left(\frac{a\tau+b}{c\tau+d}\right)=\sum_{n\in\Z} c_{\frac ac,v}(n) q^{\frac{n}{N_{\varrho}}}
\]
the \begin{it}Fourier expansion of $F$ at $\varrho$\end{it}, where $N_{\varrho}\in\N$ is the \begin{it}cusp width\end{it} of $F$ at $\varrho$. 

If $\kappa=k+\frac{1}{2}\in \Z+\frac{1}{2}$, then the Fourier expansions of $F\in\mathcal{M}_{k+\frac{1}{2}}(\Gamma)$ at certain cusps are related to each other if $F$ lies in the \begin{it}plus space\end{it} $\mathcal{M}_{k+\frac{1}{2}}^+(\Gamma)\subseteq\mathcal{M}_{k+\frac{1}{2}}(\Gamma)$ consisting of those $F\in\mathcal{M}_{k+\frac{1}{2}}(\Gamma)$ with a Fourier expansion of the shape 
\[
F(\tau)=\sum_{\substack{n\in\Z\\ (-1)^kn\equiv 0,1\pmod{4}}} c_v(n) q^n. 
\]
Choosing $M_{\frac{1}{2}}=\left(\begin{smallmatrix}1&0\\ 2&1\end{smallmatrix}\right)$ and $M_{0}=\left(\begin{smallmatrix}0&1\\ -1&0\end{smallmatrix}\right)$, we recall the relationship between these coefficients in the special case $\Gamma=\Gamma_0(4)$. If $F\in \mathcal{M}_{k+\frac{1}{2}}^+(\Gamma_0(4))$, then \cite[Proposition 6.7]{MockModularBook} (see \cite[Proposition 3]{KohnenFourier} for the case that $F$ is a holomorphic cusp form) implies the following.
\begin{lemma}\label{lem:PlusSpaceExpansions}
If $F\in \mathcal{M}_{k+\frac{1}{2}}^+(\Gamma_0(4))$, then
\begin{align*}
c_{0,v}(n)&=\frac{1+(-1)^k}{2^{2k+1}} c_{\frac{v}{16}}(4n),\quad
c_{\frac{1}{2},v}(n)=\begin{cases} \frac{\left(\frac{2}{|n|}\right)}{2^k} c_{\frac{v}{4}}(n)&\text{if }(-1)^{k}n\equiv 1\pmod{4},\\ 0&\text{otherwise}.\end{cases}
\end{align*}
\end{lemma}

We require some operators on $\mathcal{M}_{\kappa}(\Gamma)$ as well. Define for $f(\tau)=\sum_n c_v(n)q^n$ and $d\in\N$ the \textit{$U$-operator} as $f|U_d (\tau):=\sum_n c_{\frac{v}{d}}(dn)q^n$. We let
\[
C_f(\varrho):=-\lim_{z\to 0^+} z^2 f\left(\frac{h}{k}+\frac{iz}{k}\right),\qquad g_j:=\gcd(h+kj,d).
\]
A straightforward calculation gives the following. 
\begin{lemma}\label{lem:Uopgrowth}
Suppose that $f$ is translation-invariant and that $C_f(\frac hk)$ exists for every $h\in\Z$ and $k\in\N$ with $\gcd(h,k)=1$. Then we have
\[
C_{f|U_d}\left( \frac hk \right) = \frac{1}{d} \sum_{j\pmod{d}} g_j^2C_f\left(\frac{\frac{h+kj}{g_j}}{\frac{kd}{g_j}}\right).
\]
\end{lemma}

We next describe holomorphic projection. For this, suppose that $F(\tau) = \sum_{n \in \Z} c_{v}(n) q^{n}\in\mathcal{M}_{\kappa}(\Gamma)$ for some $\kappa\geq 2$. Suppose furthermore that  $F(\tau)-P_{i\infty}(q^{-1})$ has moderate growth, where $P_{i\infty} \in \C [x]$ and that a similar condition holds as $\tau\to \Q$. Following Sturm \cite{Sturm} and Gross--Zagier \cite[Proposition 5.1, p. 288]{GrossZagier}, we define (see \cite{MOR}
 for it written in this generality) the \begin{it}holomorphic projection\end{it} of $F$
\begin{align*}
\pi_{\text{hol}}^{\text{reg}} (F) (\tau) := P_{i \infty} \left( q^{-1} \right) + \sum_{n =1}^\infty c(n) q^{n}.
\end{align*}
Here for $n \in \N$
\begin{align*}
c(n) := \frac{(4 \pi n)^{\kappa-1}}{\Gamma(\kappa-1)} \lim_{s\to 0^+} \int_0^{\infty} c_{v} (n) v ^{\kappa-2-s} e^{-4 \pi n v} d v.
\end{align*}

For $F_1\in\mathcal{M}_{\kappa_1}(\Gamma)$ and $F_2\in\mathcal{M}_{\kappa_2}(\Gamma)$ with $\kappa_1,\kappa_2 \in \frac 12\Z$, define next for $\ell\in\N_0$ the $\ell$-th {\it Rankin--Cohen bracket} 
\begin{equation*}
[F_1,F_2]_\ell := \frac{1}{(2\pi i)^\ell}\sum_{j=0}^{\ell} (-1)^j \binom{\kappa_1 + \ell -1}{\ell-j} \binom{\kappa_2 + \ell -1}{j} F_1^{(j)} F_2^{(\ell-j)}
\end{equation*}
with  $\binom{\alpha}{j}:=\frac{\Gamma(\alpha+1)}{j!\Gamma(\alpha-j+1)}$. Then $[F_1,F_2]_\ell\in\mathcal{M}_{\kappa_1+\kappa_2+2\ell}(\Gamma)$ \cite[Theorem 7.1]{Cohen}.

We next define a special class of non-holomorphic modular forms known as harmonic Maass forms. The \emph{weight $\kappa$ hyperbolic Laplace operator} is defined by 
\[
\Delta_{\kappa}:=-v^2\left(\frac{\partial^2}{\partial u^2}+\frac{\partial^2}{\partial v^2}\right)+i\kappa v\left(\frac{\partial}{\partial u}+i\frac{\partial}{\partial v}\right).
\]
We call $F\in\mathcal{M}_{\kappa}(\Gamma)$ a \begin{it}harmonic Maass form of weight $\kappa$ on $\Gamma$\end{it} if $\Delta_{\kappa}(F)=0$ and if there exists $a\in\R$ such that
\begin{equation*}\label{eqn:fgrowth}
F(\tau)=O\left(e^{a v}\right)\text{ as }v\to \infty\qquad\text{ and }\qquad F(u+iv)=O\left(e^{\frac{a}{v}}\right)\text{ for }u\in\Q\text{ as }v\to 0^+.
\end{equation*}

\subsection{Generating functions related to modular forms}

We require certain generating functions together with their modular properties and their growth towards the cusps.

The \begin{it}class number generating function\end{it} is defined by 
\[
\mathcal{H}(\tau):=\sum_{n=0}^\infty H(n) q^n.\qquad 
\]
The modularity of the class number generating function is given in \cite[Theorem 2]{HZ}. To recall the statement, define the \begin{it}incomplete gamma function\end{it} (for $y>0$ and $s\in\C$)
\[
\Gamma(s,y):=\int_{y}^{\infty} t^{s-1}e^{-t}dt.
\]
\begin{theorem}\label{thm:Hcomplete}
The function 
\begin{equation*}
\widehat{\mathcal{H}}(\tau):=\mathcal{H}(\tau) +\frac{1}{8\pi \sqrt{v}}+ \frac{1}{4\sqrt{\pi}}\sum_{n=1}^\infty n\Gamma\left(-\frac12, 4\pi n^2 v\right)q^{-n^2}
\end{equation*}
is a harmonic Maass  form of weight $\frac{3}{2}$ on $\Gamma_0(4)$.
\end{theorem}
Lemma \ref{lem:PlusSpaceExpansions} gives the following expansion for $\widehat{\mathcal{H}}$ at the cusps $0$ and $\frac 12$.
\begin{corollary}\label{cor:ClassNumberExpansions}

The Fourier expansion of $\widehat{\mathcal{H}}$ at the cusp $0$ is given by  
\[
\widehat{\mathcal{H}}_0(\tau) =\frac{1-i}{8} \sum_{n=0}^\infty H(4n)q^{\frac{n}{4}}+ \frac{1-i}{16\pi\sqrt{v}} + \frac{1-i}{8\sqrt{\pi}} \sum_{n=1}^{\infty} n \Gamma\left(-\frac 12,4\pi n^2v\right) q^{-4n^2}.
\]
The Fourier expansion of $\widehat{\mathcal{H}}$ at the cusp $\frac 12$ is given by
\[
\widehat{\mathcal{H}}_{\frac 12}(\tau) = \frac 12 \sum_{n=0}^\infty \left(\frac{2}{4n+3}\right) H(4n+3) q^{n+\frac 34}+\frac{1}{8\sqrt{\pi}} \sum_{n=0}^\infty (2n+1) \Gamma\left(-\frac 12,\pi (2n+1)^2v\right) q^{-\frac{(2n+1)^2}{4}}.
\]
\end{corollary}
We next compute the growth of $\widehat{\mathcal{H}}$ towards the cusps.
\begin{lemma}\label{lem:Hcuspgrowth}
Let $h\in\Z$ and $k\in\N$ with $\gcd(h,k)=1$. Then 
\[
\lim_{z\to 0^+}z^{\frac{3}{2}}\widehat{\mathcal{H}}\left(\frac hk+\frac{iz}{k}\right)=\begin{cases} 
\frac{1}{48 \sqrt{2}}\left(\frac{h}{k}\right)\varepsilon_{k}^{-1}&\text{if $k$ is odd},\\
0&\text{if }k\equiv 2\pmod{4},\\
\frac{i^{-\frac{1}{2}}}{12}\varepsilon_{h}\left(\frac{k}{h}\right)&\text{if }4\mid k.
\end{cases}
\]
\end{lemma}

We also require certain theta functions. For $\kappa\in\N_0$, $m\in\Z$, and $M\in\N$, define 
\[
\theta_{\kappa,m,M}(\tau):=\sum_{\substack{n\in\Z\\ n \equiv m \pmod M}} n^{\kappa} q^{n^2}.
\]
 A direct calculation then gives
\begin{align*}
	 \left(\mathcal{H}\theta_{\kappa,m,M}\right)\big|U_4(\tau) 
	=
	\sum_{n= 0}^\infty H_{\kappa,m,M}(n) q^n.
\end{align*}
We require the growth of  $\theta_{m,M}$ towards all of the cusps. Setting $e_c(x):=e^{\frac{2\pi i x}{c}}$, a  standard calculation using the modular properties of $\theta_{m,M}$ (see for example \cite[Section 2]{Shimura}) yields the following.
\begin{lemma}\label{lem:Thetacuspgrowth}
Let $h,m\in\Z$ and $M,k\in\N$ be given with $\gcd(h,k)=1$. We have
\[
\lim_{z\to 0^+} \sqrt{z} \theta_{m,M}\left(\frac{h}{k}+\frac{iz}{k}\right) = \frac{e_k\!\left(h m^2\right)}{M \sqrt{2k}}G\left(hM^2,2hmM;k\right). 
\]
\end{lemma}

We next recall how to obtain formulas for moments by studying Rankin--Cohen brackets between the class number generating function and $\theta_{m,M}$. Define (compare with \cite[(7.7)]{Me}, although the notation is different there) 
\begin{equation*}
G_{k,m,M}(n) := \sum_{\substack{ t\in\Z\\ t\equiv m\pmod{M}}}p_{2k}(t,n) H\left(4n-t^2\right),
\end{equation*}
where $p_{2k}(t,n)$ denotes the $(2k)$-th coefficient in the Taylor expansion of $(1-tX+nX^2)^{-1}$. 

The numbers $G_{k,m,M}$ appear in the Fourier expansion of the Rankin--Cohen bracket $[\mathcal{H}, \theta_{m,M}]_{k} |U_4$ (see \cite[Lemma 4.1]{BKP}).
\begin{lemma}\label{lem:GcoeffRankinCohen}
The $n$-th Fourier coefficient of $[\mathcal{H}, \theta_{m,M}]_{k} |U_4$ equals \(\frac{(2k)!}{2\cdot k!} G_{k,m,M}(n)\). 
\end{lemma}
This leads to the following lemma (see \cite[Lemma 3.2]{BKP}).
\begin{lemma}\label{lem:HrelateG}
For $m\in\Z$ and $k,M\in\N$, we have 
\begin{equation*}
H_{2k,m,M}(n) = \frac{k!}{(2k)!} G_{k,m,M}(n) - \sum_{\mu=1}^{k} (-1)^\mu \frac{(2k-\mu)!}{\mu!(2k-2\mu)!} n^\mu H_{2k-2\mu,m,M}(n).
\end{equation*}
\end{lemma}

\subsection{Generalized quadratic Gauss sums}
Define for $a,b \in \Z$, and $c\in\N$ 
\begin{equation*}
G(a,b;c):=\sum_{\ell\pmod{c}}e_c\!\left(a\ell^2+b\ell\right),
\end{equation*}
 the \begin{it}generalized quadratic Gauss sum\end{it}. To evaluate these, we require some well-known properties of the Gauss sums (see \cite{BerndtEvansWilliams} for background material). For $\gcd(a,b)=1$, we let $[a]_b$ be the multiplicative inverse of $a \pmod b$.

\begin{lemma}\label{lem:Gausseval}
Suppose that $a,b\in \Z$ and $c,d\in\N$.

\begin{enumerate}[leftmargin=*, label=\rm(\arabic*)]
\item If $\gcd(a,c)\nmid b$, then $G(a,b;c)=0$ and otherwise
\begin{equation*}
G(a,b;c)=gG\left(\frac{a}{g},\frac{b}{g};\frac{c}{g}\right).
\end{equation*}

\item If $\gcd(c,d)=1$, then 
\[
G(a,b;cd)=G(ad,b;c)G(ac,b;d).
\]

\item 
If $\gcd(a,c)=1$ and $c$ is odd, then 
\begin{equation*}
G(a,b;c)=\varepsilon_c\left(\frac{a}{c}\right)e_c\!\left(-[4a]_cb^2\right)\sqrt{c}.
\end{equation*}

\item 
Suppose that $a\in\Z$ is odd and $\beta\in\N$. 
\noindent

\noindent
\textnormal{(i)}  
If $b$ is even, then we have 
\[
G\left(a,b;2^{\beta}\right)=e_{2^{\beta}}\!\left(-[a]_{2^{\beta}} \left(\frac{b}{2}\right)^2\right) G\left(a,0;2^{\beta}\right).
\]
\textnormal{(ii)} We have 
\[
G\left(a,0;2^{\beta}\right) = \begin{cases} 
0&\text{if }\beta=1,\\
(1+i)\varepsilon_{a} \left(\frac{2^{\beta}}{a}\right)2^{\frac{\beta}{2}}& \text{if }\beta\geq 2.
\end{cases}
\]
\textnormal{(iii)}  
Suppose that $b$ is odd. Then
\[
G\left(a,b;2^{\beta}\right)=\begin{cases} 2&\text{if }\beta=1,\\ 0&\text{otherwise}.\end{cases}
\]
\end{enumerate}
\end{lemma}

\section{Growth of modular forms towards the cusps}\label{sec:setup}

In this section, we determine the growth of $(\widehat{\mathcal{H}}\theta_{m,M})\big|U_{4}$  towards the cusps. For this, let $h\in\Z$ and $k\in\N$ with $\gcd(h,k)=1$. We define $c_{m,M}(h,k)$ as follows. If $k$ is even, then set 
\[
c_{m,M}(h,k):=\frac{i^{\frac{3}{2}}}{96M\sqrt{2k}}\sum_{j\pmod{4}}   \varepsilon_{h+kj} \left(\frac{k}{h+kj}\right) e_k\!\left(h m^2\right) G\left((h+kj)M^2,2(h+kj)mM;4k\right).
\]
If $k\equiv 1\pmod{4}$, then define 
\begin{multline*}
c_{m,M}(h,k):= -\frac{e_k\!\left([4]_{k}hm^2\right)}{24M\sqrt{k}} \left(\frac{h}{k}\right)\varepsilon_{k}^{-1} G\left(\frac{h-hk^2}{4}M^2,2\frac{h-hk^2}{4}mM;k\right)\\
+\frac{i^{\frac{3}{2}}e_k\!\left([4]_khm^2\right)}{96M\sqrt{2k}}  \left(\frac{k}{h}\right)\sum_{j\in\{1,3\}} i^{m^2j}  \varepsilon_{j}  G\left(\left(h-hk^2+ k^2 j\right)M^2,2\left(h-hk^2+ k^2 j\right)mM;4k\right).
\end{multline*}
Finally, for $k\equiv 3\pmod{4}$ we set 
\begin{multline*}
c_{m,M}(h,k):=-\frac{e_k\!\left([4]_{k}hm^2\right)}{24M\sqrt{k}} \left(\frac{h}{k}\right)\varepsilon_{k}^{-1} G\left(\frac{h-hk^2}{4}M^2,2\frac{h-hk^2}{4}mM;k\right)\\
+\frac{i^{\frac{3}{2}}e_k\!\left([4]_khm^2\right)}{96M\sqrt{2k}} \left(\frac{-k}{h}\right) \sum_{j\in\{1,3\}}i^{-m^2j}   \varepsilon_{j}^{-1}  G\left(\left(h-hk^2+ k^2 j\right)M^2,2\left(h-hk^2+ k^2 j\right)mM;4k\right).
\end{multline*}
Combining Lemmas \ref{lem:Hcuspgrowth}, \ref{lem:Thetacuspgrowth}, and \ref{lem:Uopgrowth} we obtain the following.

\begin{lemma}\label{lem:HthetaU4Growth}
For $h\in\Z$ and $k\in\N$ with $\gcd(h,k)=1$, we have
\[
C_{\left( \theta_{m,M}\mathcal{H}\right)|U_4}\left(\frac hk\right)=c_{m,M}(h,k).
\]
\end{lemma}
As a corollary, we obtain a precise version of \cite[Theorem 1.2]{Me}. 
Setting $\Gamma_{N,M}:=\Gamma_0(N)\cap\Gamma_1(M)$ for $M\mid N$, let 
$\mathcal{S}_{M}$ be a set of representatives for the cusps of $\Gamma_{4M^2,M}$. 
Using the usual construction via the trace from $\Gamma(4M^2)$ to $\Gamma_{4M^2,M}$ of the weight two harmonic Eisenstein series on $\Gamma(4M^2)$ appearing in \cite[(2)]{PeFieldCoeffEisen}, it is well-known that for
 each $\frac{h}{k}\in \mathcal{S}_M$ 
there exists a weight two harmonic Eisenstein series $E_{2,M,\frac{h}{k}}$ 
 satisfying for every $\frac{a}{c}\in\mathcal{S}_M$ 
\begin{equation}\label{eqn:E2Mcusp}
C_{E_{2,M,\frac{h}{k}}} \left( \frac ac \right)
=\begin{cases} 1&\text{if }\frac{a}{c}=\frac{h}{k},\\ 0&\text{otherwise}.\end{cases}
\end{equation}
We define a weight two harmonic Eisenstein series 
on $\Gamma_{4M^2,M}$ by
\begin{equation}\label{eqn:hatEmMdef}
\widehat{E}_{m,M}:=\sum_{\frac{h}{k}\in \mathcal{S}_M} c_{m,M}(h,k) E_{2,M,\frac{h}{k}}.
\end{equation}
We let $E_{m,M}$ denote the holomorphic part of $\widehat{E}_{m,M}$, so that 
$$
\widehat{E}_{m,M}(\tau)=E_{m,M}(\tau)+\frac{c}{v}
$$
for some $c\in\C$ (see \cite[Lemma 4.3]{MockModularBook}, for example). These Eisenstein series naturally occur when computing the holomorphic projection of $(\theta_{m,M}\widehat{\mathcal{H}})\big|U_4$.
\begin{corollary}\label{cor:Eisensteinpartholproj}
The function
\[
\pihol\left(\left(\theta_{m,M}\widehat{\mathcal{H}}\right)\big|U_4\right)-E_{m,M}
\]
is a holomorphic cusp form of weight two on $\Gamma_{4M^2,M}$.
\end{corollary}

\begin{proof}
Taking $g=1$ in \cite[Lemma 4.4]{Me}, we have 
\[
\pihol\left(\left(\theta_{m,M}\widehat{\mathcal{H}}\right)\big|U_4\right)-E_{m,M}= \pihol\left(\left(\theta_{m,M}\widehat{\mathcal{H}}\right)\big|U_4-\widehat{E}_{m,M}\right). 
\]
By Lemma \ref{lem:HthetaU4Growth}, the non-holomorphic modular form
\[
\left( \theta_{m,M}\widehat{\mathcal{H}}\right)\big|U_4-\widehat{E}_{m,M}
\]
decays towards all cusps. Thus 
\[
\pihol\left(\left( \theta_{m,M}\widehat{\mathcal{H}}\right)\big|U_4-\widehat{E}_{m,M}\right)
\]
is a cusp form by (the well-known extension to congruence subgroups of) \cite[Section 6.2]{ZagierUtrecht}. The level of the cusp form is the same as the level of the non-holomorphic modular form, which was determined in the proof of \cite[Theorem 1.2]{Me} (see also \cite[Lemma 3.3]{BKP}, where some constants were corrected). Since the constant term does not change when applying holomorphic projection, the Eisenstein series (including the multiple of $\widehat{E}_2$) is precisely the one appearing here.
\end{proof}

Plugging Corollary \ref{cor:Eisensteinpartholproj} into the proof of \cite[Theorem 1.2]{Me} (see also \cite[Lemma 3.3]{BKP}, which corrects an error two paragraphs before \cite[Proposition 7.2]{Me}) yields a more precise version of \cite[Theorem 1.2]{Me}. To describe this, we set 
\begin{align*}
\Lambda_{\ell,m,M}(\tau) &:= \sum_{n=1}^{\infty} \lambda_{\ell,m,M}(n) q^n,\qquad \text{where} \qquad
\lambda_{\ell,m,M}(n):=\sum_{\pm} \sideset{}{^*}\sum_{\substack{0\leq s<t\\ t^2-s^2=n \\ t\equiv \pm m\pmod{M}}}(t-s)^{\ell}.
\end{align*}
Here and throughout $\sum^*$ means that the terms in the sum with $s=0$ are weighted by $\frac{1}{2}$.

\begin{corollary}\label{cor:Mertensholproj}
For $k\in\N_0$, $m\in\Z$, and $M\in\N$, the function 
\begin{equation*}
g_{k,m,M}:=\left(\left[\mathcal{H},\theta_{m,M}\right]_k +2^{-1-2k}\binom{2k}{k}\Lambda_{2k+1,m,M}\right)\bigg|U_4-\delta_{k=0} E_{m,M}
\end{equation*}
is a holomorphic cusp form of weight $2+2k$ on $\Gamma_{4M^2,M}$.
\end{corollary}

We next evaluate the Gauss sums to obtain a more precise formula for $c_{m,M}(h,k)$ (and hence $E_{m,M}$) and also to speed up the calculation of the image under the $U$-operator using Lemma \ref{lem:Uopgrowth}. In particular, in order to use a computer to show identities, one must do precise calculations in a number field, and the speed of these calculations depends heavily on the degree of the field extension; the calculations below generally reduce the degree of the extension significantly. 

\begin{lemma}\label{lem:expGeval}
Let $h,m\in\Z$ and $k,M\in\N$ and write $M=2^{\alpha}M_0$, $k=2^{\beta}g_1g_2 k_0$, and $m=2^{\gamma}m_0$ with $M_0$, $k_0$, and $m_0$ odd. Assume that $\gcd(h,k)=1$ and $\alpha\ge\gamma$, and set $g_1:=\gcd(M,k)$, $g_2:=\gcd(M,\frac{k}{g_1})$. If $g_2\nmid 2m$, then
\[
G\left(hM^2,2hmM;k\right)=0.
\]
If $g_2\mid 2m$, then 
\begin{multline*}
e_k\!\left(hm^2\right) G\left(hM^2,2hmM;k\right)\\
=\sqrt{g_1g_2 k}
\begin{cases}
\vspace{.2cm} \varepsilon_{k_0}\left(\frac{h\frac{M^2}{g_1g_2}}{k_0}\right)e_{4\frac{g_1}{g_2}}\!\left( h \left(\frac{2m}{g_2}\right)^2\left[k_0\right]_{\frac{4g_1}{g_2}}\right)&\text{if }\beta=0,\\
\sqrt{2}\varepsilon_{k_0} \left(\frac{2h\frac{M^2}{g_1g_2}}{k_0}\right)e_{8\frac{g_1}{g_2}}\!\left( h \left(\frac{2m}{g_2}\right)^2\left[k_0\right]_{\frac{8g_1}{g_2}}\right)&\text{if }\beta =1\text{ and }\gamma=\alpha-1,\\
0&\text{if }\beta=1\text{ and }\gamma=\alpha,\\
0&\text{if }\beta\geq 2\text{ and }\gamma=\alpha-1,\\
(1+i)(-1)^{\frac{k_0-1}{2}}\varepsilon_{\frac{hM^2}{g_1g_2}}\left(\frac{2^{\beta} k_0}{h\frac{M^2}{g_1g_2}}\right)e_{\frac{g_1}{g_2}}\!\left( h \left(\frac{m}{g_2}\right)^2\left[2^{\beta}k_0\right]_{\frac{g_1}{g_2}}\right)&\text{if }\beta\geq 2\text{ and }\gamma=\alpha.
\end{cases}
\rm
\end{multline*}
\end{lemma}
\begin{proof}
Using the fact that $\gcd(h,k)=1$, we have 
\begin{equation}\label{eqn:gcd(a,c)}
\gcd\left(hM^2,k\right)=g_1g_2.
\end{equation}
Hence Lemma \ref{lem:Gausseval} (1) implies that the Gauss sum vanishes if $g_2\nmid 2m$.

Now assume that $g_2\mid 2m$. By Lemma \ref{lem:Gausseval} (1) with $g=g_1g_2$ we obtain
\begin{equation}\label{eqn:Gaussgcdpluggedin}
G\left(hM^2,2hmM;k\right)=g_1g_2G\left(h\frac{M^2}{g_1g_2},h\frac{2m}{g_2}\frac{M}{g_1};\frac{k}{g_1g_2}\right).
\end{equation}
Note that $\gcd(h\frac{M^2}{g_1g_2},\frac{k}{g_1g_2})=1$ by \eqref{eqn:gcd(a,c)}. We put $c=2^{\beta}$ and $d=k_0$ in Lemma \ref{lem:Gausseval} (2) to obtain 
\begin{equation}\label{eqn:Gauss2powsplit}
G\left(h\frac{M^2}{g_1g_2},h\frac{2m}{g_2}\frac{M}{g_1};\frac{k}{g_1g_2}\right)=G\left(2^{\beta} h\frac{M^2}{g_1g_2},h\frac{2m}{g_2}\frac{M}{g_1};k_0\right)G\left(k_0 h\frac{M^2}{g_1g_2},h\frac{2m}{g_2}\frac{M}{g_1};2^{\beta}\right).
\end{equation}
Since $k_0$ is odd, Lemma \ref{lem:Gausseval} (3) implies that
\begin{equation*}
G\left(2^{\beta}h\frac{M^2}{g_1g_2},h\frac{2m}{g_2}\frac{M}{g_1};k_0\right) = \varepsilon_{k_0}\sqrt{k_0} \left(\frac{2^{\beta}h\frac{M^2}{g_1g_2}}{k_0}\right)e_{k_0}\!\left(-\left[2^{\beta +2}h \frac{M^2}{g_1g_2}\right]_{k_0}\left(h\frac{2m}{g_2}\frac{M}{g_1}\right)^2\right).
\end{equation*}
Simplifying yields (note that $e_{c}(x)=e_{ac}(ax)$)
\begin{multline}\label{eqn:Gaussoddpart} 
G\left(2^{\beta}h\frac{M^2}{g_1g_2},h\frac{2m}{g_2}\frac{M}{g_1};k_0\right)
\\ = \varepsilon_{k_0}\sqrt{k_0} \left(\frac{2^{\beta}h\frac{M^2}{g_1g_2}}{k_0}\right)e_{2^{\beta+2}\frac{g_1}{g_2}k_0}\!\left(-h\left(\frac{2m}{g_2}\right)^2  2^{\beta+2}\frac{g_1}{g_2} \left[2^{\beta+2}\frac{g_1}{g_2}\right]_{k_0}\right).
\end{multline}
Recalling that $k=2^{\beta}k_0g_1g_2$ and $g_2\mid g_1$, we similarly rewrite
\begin{equation}\label{eqn:exponentialrewrite}
e_k\!\left(hm^2\right) = e_{2^{\beta}g_1g_2 k_0}\!\left( hm^2\right) = e_{2^{\beta+2}\frac{g_1}{g_2}k_0}\!\left( h\left(\frac{2m}{g_2}\right)^2\right).
\end{equation}

Now suppose that $\beta=0$. Plugging \eqref{eqn:Gaussoddpart} into \eqref{eqn:Gaussgcdpluggedin}, using \eqref{eqn:exponentialrewrite}, and then simplifying yields
\begin{equation*}
e_k\!\left(h m^2\right) G\left(hM^2,2hmM;k\right)= \varepsilon_{k_0} g_1g_2\sqrt{k_0} \left(\frac{h\frac{M^2}{g_1g_2}}{k_0}\right)e_{4 \frac{g_1}{g_2}}\!\left(h \left(\frac{2m}{g_2}\right)^2 \left[k_0\right]_{4\frac{g_1}{g_2}}\right).
\end{equation*}

\rm

Next suppose that $\beta>0$. In this case $h\frac{M^2}{g_1g_2}$ is odd because $\gcd(h\frac{M^2}{g_1g_2},\frac{k}{g_1g_2})=1$.  Thus 
\begin{equation*}
\ord_2(g_1g_2)=2\ord_2(M)=2\alpha.
\end{equation*}
 Since $\alpha\ge\gamma$ (by assumption), we have that
\[
0\leq \ord_{2}\left(h\frac{2m}{g_2}\frac{M}{g_1}\right)=1+\gamma+\alpha-\ord_2\left(g_1g_2\right)= 1+\gamma-\alpha\leq 1.
\]
Thus $\gamma=\alpha$ or $\gamma=\alpha-1$.
Moreover, Lemma \ref{lem:Gausseval} (4) implies that 
\begin{multline*}
G\left(k_0 h\frac{M^2}{g_1g_2},h\frac{2m}{g_2}\frac{M}{g_1};2^{\beta}\right)\\
=
\begin{cases}
0&\text{if }\beta\geq 2,\, \gamma=\alpha-1,\\
2&\text{if }\beta=1,\, \gamma=\alpha-1,\\
0&\text{if }\beta=1,\,\gamma=\alpha,\\
(1+i)\varepsilon_{k_0h\frac{M^2}{g_1g_2}}2^{\frac{\beta}{2}} \left(\frac{2^{\beta}}{k_0h\frac{M^2}{g_1g_2}}\right)e_{2^{\beta}}\!\left(-hm_0^2\left[\frac{g_1g_2}{2^{2\alpha}}k_0\right]_{2^{\beta}}\right) & \text{if }\beta\geq 2,\, \gamma=\alpha.
\end{cases}
\end{multline*}
The cases ($\beta\geq 2$ and $\gamma=\alpha-1$) and ($\beta=1$ and $\gamma=\alpha$) now follow immediately from \eqref{eqn:Gauss2powsplit}, while the case $\beta=1$ and $\gamma=\alpha-1$ is implied by the Chinese Remainder Theorem after simplifying.

Finally, for $\beta\geq 2$ and $\gamma=\alpha$ (so $m_0=\frac{m}{2^{\alpha}}= \frac{2m}{g_2} \frac{g_2}{2^{\alpha+1}}$), we combine with \eqref{eqn:Gaussoddpart} and plug back into \eqref{eqn:Gauss2powsplit}. Using \eqref{eqn:exponentialrewrite} and plugging into \eqref{eqn:Gaussgcdpluggedin}, we obtain 
\begin{multline*}
e_k\!\left(hm^2\right) G\left(hM^2,2hmM;k\right)=g_1g_2(1+i) \varepsilon_{k_0}\varepsilon_{k_0h\frac{M^2}{g_1g_2}} \sqrt{2^{\beta}k_0}\left(\frac{2^{\beta}h\frac{M^2}{g_1g_2}}{k_0}\right) \left(\frac{2^{\beta}}{k_0h\frac{M^2}{g_1g_2}}\right)\\
\times e_{2^{\beta+2}\frac{g_1}{g_2}k_0}\!\Bigg(h \left(\frac{2m}{g_2}\right)^2\Bigg(1  -2^{\beta+2}\frac{g_1}{g_2}\left[2^{\beta+2}\frac{g_1}{g_2}\right]_{k_0}  -  \frac{g_1g_2}{2^{2\alpha}}k_0\left[\frac{g_1g_2}{2^{2\alpha}}k_0\right]_{2^{\beta}}\Bigg)\Bigg).
\end{multline*}
We then again use the Chinese Remainder Theorem to simplify the exponential, yielding in the case $\beta\geq 2$ and $\gamma=\alpha$ that 
\begin{multline*}
e_{k}\!\left(hm^2\right) G\left(hM^2,2hmM;k\right)\\
= g_1g_2(1+i)\varepsilon_{k_0}\varepsilon_{k_0h\frac{M^2}{g_1g_2}} \sqrt{2^{\beta}k_0}\left(\frac{2^{\beta}h\frac{M^2}{g_1g_2}}{k_0}\right) \left(\frac{2^{\beta}}{k_0h\frac{M^2}{g_1g_2}}\right) e_{\frac{g_1}{g_2}}\!\left(h \left(\frac{m}{g_2}\right)^2\left[2^{\beta}k_0\right]_{\frac{g_1}{g_2}}\right).
\end{multline*}

Via a short calculation, we may finally simplify the factors in front as
\[
\varepsilon_{k_0}\varepsilon_{k_0h\frac{M^2}{g_1g_2}} \left(\frac{2^{\beta}h\frac{M^2}{g_1g_2}}{k_0}\right) \left(\frac{2^{\beta}}{k_0h\frac{M^2}{g_1g_2}}\right)=(-1)^{\frac{k_0-1}{2}}\varepsilon_{h\frac{M^2}{g_1g_2}}\left(\frac{2^{\beta} k_0}{h\frac{M^2}{g_1g_2}}\right).\qedhere
\]
\end{proof}

\section{Proof of Theorem \ref{prop:H1,m,3}}\label{sec:explicitmoments}

\subsection{Explicit formulas for sums of class numbers}

Since Lemma \ref{lem:HrelateG} gives a recursive formula for moments, in order to obtain an explicit formula for higher moments, we first require explicit formulas for sums of class numbers (the zeroeth moments) $H_{m,M}(n)$, in particular for the case $M=3$.

\begin{lemma}\label{lem:HmM}
\ \begin{enumerate}[leftmargin=*, label=\rm(\arabic*)]
\item For $n\in\N$ with $3\nmid n$, we have 
\begin{equation*}\label{eqn:H03}
H_{0,3}(n)=\begin{cases}\frac{1}{2}\sigma(n)&\text{if }n\equiv 1\pmod{3},\\
\sigma(n)-2\sum_{\substack{d\mid n, d^2<n}} d&\text{if }n\equiv 2\pmod{3}.
\end{cases}
\end{equation*}

\item For $n\in\N$ with $3\nmid n$, we have 
\begin{equation*}\label{eqn:H13}
H_{1,3}(n)=H_{2,3}(n)=
\begin{cases}\frac{3}{4}\sigma(n)-\frac{1}{2}\sum_{d\mid n} \min\left(d,\frac{n}{d}\right)&\text{if }n\equiv 1\pmod{3},\\
\frac{1}{2}\sigma(n)&\text{if }n\equiv 2\pmod{3}.
\end{cases}
\end{equation*}

\item For $n\in\N$ we have
\begin{equation*}\label{eqn:H03(3n)}
H_{0,3}(3n)=2\sigma(n)-6\delta_{3\mid n}\sum_{d\mid \frac{n}{3},d^2<\frac{n}{3}} d -\delta_{12n=\square}\sqrt{3n}.
\end{equation*}

\item For $n\in\N$, we have 
\[
\hspace{.15cm} H_{1,3}(3n)=H_{2,3}(3n)=\sigma(3n)-\frac{1}{2}\sum_{d\mid 3n}\min\left(d,\tfrac{3n}{d}\right)-\sigma(n)+3\delta_{3\mid n}\sum_{d\mid \frac{n}{3}, d^2<\frac{n}{3}} d +\frac{1}{4}\delta_{12n=\square}\sqrt{12n}. 
\]
\end{enumerate}
\end{lemma}

\begin{proof}
\noindent

\noindent
(1) Noting that for $d\mid  n$ with $3\nmid n$ the condition $3\nmid d$ is automatically satisfied, the claim follows immediately from \cite[(4.2)]{BKClassNum}.

\noindent
(2) By the symmetry $t \mapsto -t$ in \eqref{eqn:HmMkdef} we have $H_{2k,1,3}(n)=H_{2k,2,3}(n)$ and hence 
\begin{equation}\label{eqn:H2k}
H_{2k,0,1}(n)=H_{2k,0,3}(n)+H_{2k,1,3}(n)+H_{2k,2,3}(n)=H_{2k,0,3}(n)+2H_{2k,1,3}(n).
\end{equation}
Since (see for example \cite[Section 7.2, Example 2]{DMZ}) 
\[
H_{0,1}(n)=2\sigma(n)-\sum_{d\mid n} \min\left(d,\frac{n}{d}\right), 
\]
we have from \eqref{eqn:H2k} with $k=0$ followed by part (1) that
\begin{align}
\label{eqn:H012}
H_{1,3}(n)&=\frac{1}{2}\left(H_{0,1}(n)-H_{0,3}(n)\right)\\
\nonumber & = \begin{cases}\frac{3}{4}\sigma(n)-\frac{1}{2}\sum_{d\mid n} \min\left(d,\frac{n}{d}\right)&\text{if }n\equiv 1\pmod{3},\\
\frac{1}{2}\sigma(n)-\frac{1}{2}\sum_{d\mid n} \min\left(d,\frac{n}{d}\right)+\sum_{d\mid n, d^2<n} d&\text{if }n\equiv 2\pmod{3}.
\end{cases}
\end{align}
Noting that $n\equiv 2\pmod{3}$ implies that $n$ is not a square, we have for $n\equiv 2\pmod{3}$ 
\[
-\frac{1}{2}\sum_{d\mid n} \min\left(d,\frac{n}{d}\right)+\sum_{d\mid n, d^2<n} d=-\sum_{d\mid n,d^2<n}d + \sum_{d\mid n,d^2<n} d =0,
\]
yielding the claim.

\noindent
(3) We first show that 
\begin{equation}\label{eqn:H03div3coeffs}
\left(\mathcal{H}\theta_{0,3}\right)\big|U_{12}+\frac{1}{2}\Lambda_{1,0,3}\big|U_{12}=  -\frac{1}{12}E_2.
\end{equation}
For this, we apply $U_3$ to the function appearing in Corollary \ref{cor:Mertensholproj} in the case $k=0$, $m=0$, and $M=3$. Since the group generated by $\Gamma_{36,3}$ and $-I$ is $\Gamma_0(36)$, we obtain that $g_{0,3}|U_3$ is a holomorphic cusp form of weight two on $\Gamma_{0}(36)$ (by \cite[Proposition 2.22]{OnoBook}, the level is preserved under the action of $U_3$ because $3\mid 36$). Evaluating the constants in \eqref{eqn:hatEmMdef} and the definition \eqref{eqn:E2Mcusp} with a computer, one easily checks that for all $\frac hk\in\Q$
\[
C_{\widehat{E}_{0,3}|U_3} \left(\frac hk \right) =-\frac{1}{12}.
\]
This immediately implies that
\begin{equation}\label{eqn:Eis03}
E_{0,3}\big|U_{3}= -\frac{1}{12}E_2.
\end{equation}

Plugging \eqref{eqn:Eis03} into the definition of $g_{0,3}|U_3$, we see that \eqref{eqn:H03div3coeffs} is equivalent 
(note that $[f,g]_0=fg$) to the claim that the cusp  form $g_{0,3}|U_3$ vanishes identically. Due to the valence formula, for this, we need to check that the first 
\[
\frac{2}{12}\left[\SL_2(\Z):\Gamma_0(36)\right] = \frac{1}{6}36\left(1+\frac{1}{2}\right)\left(1+\frac{1}{3}\right) = 12
\]
coefficients vanish. This is easily checked with a computer.

We next rewrite $\lambda_{1,0,3}(12n)$ as a divisor sum. By definition, we have 
\[
\lambda_{1,0,3}(12n)=\sum_{\pm} \sideset{}{^*}\sum_{\substack{0\leq s<t\\ t^2-s^2=12n \\ t\equiv \pm 0\pmod{3}}}(t-s).
\]
We then set $d:=t-s$ and note that $t^2-s^2=12n$ implies that $\frac{12n}{d}=t+s$ and since $s\geq 0$ we conclude that $d\leq\frac{12n}{d}$. Moreover, since $t=\frac{1}{2}(d+\frac{12n}{d})$ and $s=\frac{1}{2}(\frac{12n}{d}-d)$, there is a one-to-one correspondence between pairs $0\leq s<t$ and divisors $d\mid 12n$ for which $\frac{1}{2}(d+\frac{12n}{d})\equiv 0\pmod{3}$, or equivalently $d+\frac{12n}{d}\equiv 0\pmod{6}$. We may hence write (note that if $12n$ is a square, then $3\mid \sqrt{12n}$ is automatically satisfied) 
\begin{equation}\label{eqn:lambda103(12n)}
\lambda_{1,0,3}(12n)=2\sum_{\substack{d\mid 12n, d^2<12n\\ d+\frac{12n}{d}\equiv 0\pmod{6}}} d + \delta_{12n=\square}\sqrt{12n}.
\end{equation}
Next note that for every divisor $d\mid 12n$, we have $3\mid d$ or $3\mid \frac{12n}{d}$, so $d+\frac{12n}{d}\equiv 0\pmod{3}$ if and only if $3\mid d$ and $3\mid \frac{12n}{d}$. If such a divisor exists, then $9\mid 12n$, which is equivalent to $3\mid n$. Similarly, for any divisor $d\mid 12n$, we have $2\mid d$ or $2\mid \frac{12n}{d}$, so $d+\frac{12n}{d}\equiv 0\pmod{2}$ if and only if $2\mid d$ and $2\mid \frac{12n}{d}$. Such a divisor exists if and only if $4\mid 12n$, which holds for all $n$. We conclude that the condition $d+\frac{12n}{d}\equiv 0\pmod{6}$ in \eqref{eqn:lambda103(12n)} is equivalent to $6\mid d$ and $6\mid \frac{12n}{d}$ and taking $d\mapsto 6d$ in \eqref{eqn:lambda103(12n)} yields
\begin{align*}
\lambda_{1,0,3}(12n)&=12\sum_{\substack{d\mid 2n, d^2<\frac{n}{3}\\ 6\mid \frac{2n}{d}}} d + \delta_{12n=\square}\sqrt{12n} =12\delta_{3\mid n}\sum_{d\mid \frac{n}{3},d^2<\frac{n}{3}} d + \delta_{12n=\square}\sqrt{12n}.
\end{align*}
Comparing the coefficients on both sides of \eqref{eqn:H03div3coeffs} and plugging this in yields the claim. 

\noindent
(4) We use \eqref{eqn:H012} and \eqref{eqn:H2k} with $k=0$ and part (3) to obtain 
\begin{align*}
H_{1,3}(3n)&=\frac{1}{2}\left(H_{0,1}(3n)-H_{0,3}(3n)\right)\\
& = \sigma(3n)-\frac{1}{2}\sum_{d\mid 3n}\min\left(d,\frac{3n}{d}\right)-\sigma(n)+3\delta_{3\mid n}\sum_{d\mid \frac{n}{3}, d^2<\frac{n}{3}} d +\frac{1}{4}\delta_{12n=\square}\sqrt{12n}. \qedhere
\end{align*}
\end{proof}

\subsection{Explicit formulas for second moments}
For small choices of $M$ and $k$, the space of cusp forms of weight $2+2k$ on  $\Gamma_{4M^2,M}$ is  relatively small (possibly empty in some cases), leading to a recursive formula for $H_{2k,m,M}(n)$ via Lemma \ref{lem:HrelateG} and Corollary \ref{cor:Mertensholproj}. Here we consider the case $M=3$ and $k=1$. We begin by explicitly determining the cusp forms $g_{1,m,3}$ occuring in Corollary \ref{cor:Mertensholproj}. To state the result, for $d\in\N$, we let $f\big|V_d(\tau):=f(d\tau)$ be the usual \textit{$V$-operator}.

\begin{lemma}\label{lem:g1m3}
We have
\begin{align*}
g_{1,0,3}(\tau)&=-\eta^8\big|V_3(\tau)=-q+8q^4-20q^7+70q^{13}-64q^{16}-56q^{19}+O\left(q^{25}\right),\\
g_{1,\pm 1,3}(\tau)&=\frac{1}{2}\eta^8\big|V_3(\tau).
\end{align*}
\end{lemma}

\begin{proof}
Using Corollary \ref{cor:Mertensholproj} and \cite[Theorem 1.64]{OnoBook} and again noting that the group generated by $\Gamma_{36,3}$ and $-I$ is $\Gamma_0(36)$, the result follows easily by the valence formula for weight four forms on $\Gamma_0(36)$.
\end{proof}

We are now ready to prove Theorem \ref{prop:H1,m,3}.

\begin{proof}[Proof of Theorem \ref{prop:H1,m,3}]
\noindent

\noindent
(1) Using Lemma \ref{lem:HrelateG} and then Lemma \ref{lem:HmM} (1), (2), we obtain
\begin{align}\label{eqn:H2m3gen}
\nonumber H_{2,m,3}(n)&=\frac{1}{2}G_{1,m,3}(n)+ n H_{m,3}(n)\\
&= \frac{1}{2}G_{1,m,3}(n)+n
\begin{cases}\frac{1}{2}\sigma(n)&\text{if }m=0,\ n\equiv 1\pmod{3},\\
\sigma(n)-2\sum_{d\mid n, d<\frac{n}{d}} d&\text{if }m=0,\ n\equiv 2\pmod{3},\\
\frac{3}{4}\sigma(n)-\frac{1}{2}\sum\limits_{d\mid n} \min\left(d,\frac{n}{d}\right)&\text{if }m\neq 0,\ n\equiv 1\pmod{3},\\
\frac{1}{2}\sigma(n)&\text{if }m\neq 0,\ n\equiv 2\pmod{3}.
\end{cases}
\end{align}

By Lemma \ref{lem:GcoeffRankinCohen}, $G_{1,m,3}(n)$ is the $n$-th Fourier coefficient of $\left[\mathcal{H},\theta_{m,3}\right]_1|U_4$. By Corollary \ref{cor:Mertensholproj} and Lemma \ref{lem:g1m3} we have 
\begin{equation}\label{eqn:RSeval}
\left[\mathcal{H},\theta_{m,3}\right]_1\big|U_4=-\frac{1}{4}\Lambda_{3,m,3}\big|U_4+g_{1,m,3}=-\frac{1}{4}\Lambda_{3,m,3}\big|U_4+\begin{cases} -\eta^8\big|V_3&\text{if }m=0,\\ \frac{1}{2}\eta^8\big|V_3&\text{if }m\neq 0.\end{cases}
\end{equation}
 We then evaluate
\begin{equation}\label{eqn:lambdageneral}
\lambda_{3,m,3}(4n)=\frac{1}{2}\delta_{n=\square}\sum_{\substack{\pm\\ 2\sqrt{n}\equiv \pm m\pmod{3}}} (4n)^{\frac{3}{2}}+ \sum_{\pm} \sum_{\substack{d\mid 4n,\ d^2<4n\\ d+\frac{4n}{d}\equiv \pm 2m\pmod{6}}} d^3.
\end{equation}
Plugging \eqref{eqn:lambdageneral} back into \eqref{eqn:RSeval} and then using \eqref{eqn:H2m3gen} yields the claim after noting that, since 
\begin{equation}\label{eqn:etaprod}
\eta(3\tau)^8=q \prod_{n=1}^{\infty} \left(1-q^{3n}\right)^8,
\end{equation}
we have $a(n)=0$ for $n\equiv 2\pmod{3}$.

\noindent
(2)  Following \eqref{eqn:H2m3gen}, we plug Lemma  \ref{lem:HmM} (3), (4) into Lemma \ref{lem:HrelateG} to obtain 
\begin{align}\label{eqn:H2m3(3n)gen}
&H_{2,m,3}(3n)=\frac{1}{2}G_{1,m,3}(3n)+ 3n H_{m,3}(3n)= \frac{1}{2}G_{1,m,3}(3n)\\
\nonumber
&+3n \begin{cases}
2\sigma(n)-6\delta_{3\mid n}\displaystyle \sum_{d\mid \frac{n}{3}, d^2<\frac{n}{3}} d-\sqrt{3n}\delta_{12n=\square}&\text{if }m=0,\\
\sigma(3n)-\sigma(n)-\frac12 \displaystyle\sum_{d\mid 3n}\min\left(d,\frac{3n}{d}\right)+3\delta_{3\mid n}\displaystyle\sum_{d\mid \frac{n}{3}, d^2<\frac{n}{3}} d +\frac{\sqrt{3n}}{2}\delta_{12n=\square}&\text{if }m\neq0.
\end{cases}
\end{align}
It remains to evaluate $G_{1,m,3}(3n)$. We follow part (1) again for this. By Lemma \ref{lem:GcoeffRankinCohen}, $G_{1,m,3}(3n)$ is the $(3n)$-th Fourier coefficient of $\left[\mathcal{H},\theta_{m,3}\right]_1|U_4$. We then use \eqref{eqn:RSeval} followed by \eqref{eqn:lambdageneral} to conclude  
\begin{align*}
 G_{1,m,3}(3n)&=-\frac{1}{4} \lambda_{3,m,3}(12n)+ \begin{cases} -a(3n)&\text{if }m=0,\\ \frac{1}{2}a(3n)&\text{if }m\neq0,\end{cases}\\
&=-\frac{1}{8}\delta_{3n=\square}\sum_{\substack{\pm\\ 2\sqrt{3n}\equiv \pm m\pmod{3}}} (12n)^{\frac{3}{2}}-\frac{1}{4} \sum_{\pm} \sum_{\substack{d\mid 12n,\ d^2<12n\\ d +\frac{12n}{d}\equiv \pm 2m\pmod{6}}} d^3+ \begin{cases} -a(3n)&\text{if }m=0,\\ \frac{1}{2}a(3n)&\text{if }m\neq0.\end{cases}
\end{align*}
Note however that by \eqref{eqn:etaprod} we have $a(3n)=0$.  Therefore 
\begin{equation}\label{eqn:lambdageneral3n}
G_{1,m,3}(3n)=-\delta_{3n=\square}\sum_{\substack{\pm\\ 2\sqrt{3n}\equiv \pm m\pmod{3}}} (3n)^{\frac{3}{2}} -\frac{1}{4}\sum_{\pm} \sum_{\substack{d\mid 12n,\ d^2<12n\\ d +\frac{12n}{d}\equiv \pm 2m\pmod{6}}} d^3.
\end{equation}
Plugging \eqref{eqn:lambdageneral3n} back into \eqref{eqn:H2m3(3n)gen} and simplifying 
\[
\frac{1}{4}\sum_{\pm} \sum_{\substack{d\mid 12n,\ d^2<12n\\ d +\frac{12n}{d}\equiv \pm 2m\pmod{6}}} d^3=
\begin{cases}
108\delta_{3\mid n}\sum_{d\mid \frac{n}{3}, d^2<\frac{n}{3}} d^3 &\text{if }m=0,\\
2\sum_{\substack{d\mid 3n, d^2<3n\\ 3\nmid d}} d^3 + 2\sum_{\substack{d\mid 3n, d^2<3n\\ \hspace{-.89cm} 3\nmid \frac{3n}{d}}} d^3&\text{if }m\neq 0,
\end{cases}
\]
yields the claim. 
\end{proof}

In some special cases, Theorem \ref{prop:H1,m,3} simplifies if $n$ is restricted so that all divisor sums can be computed. To illustrate this, we next specialize Theorem \ref{prop:H1,m,3} to the case that $n$ is a prime power. A short and direct calculation yields the following.
\begin{corollary}\label{cor:H1,m,3primepower}
Suppose that $m\in\{0,1,2\}$. Then the following hold.
\noindent

\noindent
\begin{enumerate}[leftmargin=*, label=\rm(\arabic*)]
\item 
Let $r\in\N$  and a prime $p\neq 3$ be given. If $p\equiv 1\pmod{3}$, then we have 
\[
H_{2,m,3}\left(p^r\right)=
\begin{cases} 
\frac{p^r}{2}\frac{p^{r+1}-1}{p-1}-\frac{1}{2}a\left(p^r\right)&\text{if }m=0,\\
-\frac{p^{3\left\lfloor\frac{r}{2}\right\rfloor+3}-1}{p^3-1}+   \frac{3p^r}{4}\frac{p^{r+1}-1}{p-1}-p^r\frac{p^{\left\lfloor\frac{r+1}{2}\right\rfloor}-1}{p-1}  +\frac{1}{4}a\left(p^r\right)&\text{if }m\neq 0.
\end{cases}
\]
For $p\equiv 2\pmod{3}$ we have 
\[
H_{2,m,3}\left(p^r\right)=
\begin{cases} 
\frac{p^r}{2}\frac{p^{r+1}-1}{p-1}-\frac{1}{2}a\left(p^r\right)&\text{if }m=0\text{ and }2\mid r,\\
-2\frac{p^{\frac{3r+3}{2}}-1}{p^3-1} + p^r\frac{p^{r+1}-1}{p-1}-2p^r\frac{p^{\left\lfloor\frac{r+1}{2}\right\rfloor}-1}{p-1}&\text{if }m=0\text{ and } 2\nmid r,\\
-\frac{p^{3\left\lfloor\frac{r}{2}\right\rfloor+3}-1}{p^3-1}+   \frac{3p^r}{4}\frac{p^{r+1}-1}{p-1}-p^r\frac{p^{\left\lfloor\frac{r+1}{2}\right\rfloor}-1}{p-1}  +\frac{1}{4}a\left(p^r\right)&\text{if }m\neq 0\text{ and }2\mid r,\\
\frac{p^r}{2}\frac{p^{r+1}-1}{p-1}&\text{if }m\neq 0\text{ and }2\nmid r.
\end{cases}
\]

\item 
For $r\in\N$ we have
\[
H_{2,m,3}\left(3^r\right)=
\begin{cases}
-3^{\frac{3r}{2}}\delta_{2\mid r} -\frac{27}{13}\left(3^{3\left\lfloor\frac{r-1}{2}\right\rfloor}-1\right)+  \displaystyle 3^r\left(3^r-1\right) - 3^{r+1}\left(3^{\left\lfloor\frac{r-1}{2}\right\rfloor}-1\right)-3^{\frac{3r}{2}}\delta_{2\mid r}&\text{if }m=0,\\ 
\vspace{-.4cm}\\
-1+  3^{2r} -\frac{3^r}{2}\left(3^{\left\lfloor\frac{r+1}{2}\right\rfloor}-1\right) +3^{r+1}\frac{3^{\left\lfloor\frac{r-1}{2}\right\rfloor}-1}{2}&\text{if }m\neq0.
\end{cases}
\]
\rm

\end{enumerate}
\end{corollary}

\end{document}